\documentclass[12pt,leqno]{amsart}
\usepackage[dvips]{graphics}
\usepackage{amssymb}

\numberwithin{equation}{section}

\newtheorem{thm}{THEOREM}[section]

\newtheorem{prop}[thm]{PROPOSITION}

\newtheorem{quest}[thm]{PROBLEM}

 \theoremstyle{definition}

\theoremstyle{remark}

\newcommand{\tref}[1]{Theorem~\ref{#1}}
\newcommand{\cref}[1]{Corollary~\ref{#1}}
\newcommand{\pref}[1]{Proposition~\ref{#1}}

\newcommand{\R}{\mathbb{R}}

\begin{document}
\pagebreak


\title{On contractible orbifolds}

\author{Alexander Lytchak}
\address{A. Lytchak, Mathematisches Institut, Einsteinstrasse 62,
48149 M\"unster, Germany}
\email{lytchak\@@math.uni-bonn.de}



\begin{abstract}
We prove that a contractible    orbifold is a manifold.
\end{abstract}

\thanks{ The author was supported   by a  Heisenberg grant of the DFG and by the  SFB  878
{\it Groups, Geometry and Actions}}

\maketitle
\renewcommand{\theequation}{\arabic{section}.\arabic{equation}}
\pagenumbering{arabic}


\section{Introduction}
Following \cite{Davis}, we call an orbifold $X$ \emph{contractible} if all of its orbifold homotopy groups $\pi_i ^{orb} (X), i\geq 1$ vanish.
We refer the reader to \cite{Davis} and the literature therein for basics about orbifolds. 
Davis has asked in  \cite{Davis}, whether any contractible orbifold $X$  must be developable. 
 In this note we answer this question affirmatively.

\begin{thm} \label{main}
Let $X$ be a smooth contractible orbifold. Then it is a manifold. 
\end{thm}

\begin{proof}
Since $X$ is contractible, it is  it is orientable.
Let $n$ be the dimension of $X$. Define on $X$ a Riemannian metric and
let the smooth manifold  $M$ be the bundle of oriented orthonormal frames on  $X$ (cf. \cite{Hae}). Then $G=SO(n)$ acts effectively and almost freely 
on $M$ with  $X=M/G$. 
 
Let $E$ denote a contractible CW complex on which $G$ act freely, with quotient $E/G =BG$, the classifying space of $G$.
Then $\hat X=(M\times E ) /G$ is the \emph{classifying space of $X$} (cf. \cite{Hae}). By definition, the orbifold homotopy groups of $X$ are the usual homotopy groups of $\hat X$. Thus, by our assumption, the topological space $\hat X$ is contractible.  

The projection  $M\times E \to \hat  X$ is a homotopy fibration. Thus the contractibility of $\hat X$ implies that the embedding of any orbit of $G$ into $ M\times E$
is a homotopy equivalence between $G$ and $M\times E$.  Since $E$ is contractible, the projection $M\times E \to M$ is a homotopy equivalence as well.
Therefore, for any $p\in M$, the composition $o_p: G \to G\cdot p \to M$ given by orbit map $o_p (g):=  g \cdot p$ is a homotopy equivalence.

Assume now that $X$ is not a manifold. Then $G$ does not act freely on $M$. Thus, for some $p\in M$,  the stabilizer $G_p$ of $p$ is a finite non-trivial group. Then the orbit map $o_p: G\to M$ factors through the quotient map $\pi _p: G\to G/G_p$.
Since the orbit map is a homotopy equivalence, there must exist some map $i: G/G_p \to G$ such that $i \circ \pi  :G\to G$ is a homotopy equivalence. However, the manifolds $G$ and $G/G_p$ are orientable and the map $G\to G/G_p$ is a covering of degree $|G_p|$. Thus,
for $m=\dim (G) =n (n-1) /2$,
the image of $\pi _p ^{\ast}$ in $H^m (G, \mathbb Z) =\mathbb Z$ is a subgroup of $H^m (G, \mathbb  Z)$ of index 
$|G_p|$.
In particular, $(i\circ \pi ) ^{\ast}$ cannot be surjective. Contradiction. 
\end{proof}

We add the following not  surprising improvement:

\begin{prop} \label{dimension}
Let $X$ be a smooth $n$-dimensional orbifold. If the orbifold homotopy groups $\pi_i ^{orb} (X) $ are trivial, for all
$0\leq i \leq n$, then $X$ is a contractible manifold.
\end{prop}

\begin{proof}
Let $M, G$ and $\hat X= M\times E /G$ be as in the proof of \tref{main}.  Consider the induced vector bundle
$\hat Y= \hat X \times \R ^n /G$  over the classifying space $\hat X$.  By assumption, $\hat X$ is $n$-connected,
hence the vector bundle $\hat Y \to \hat X$ is a trivial vector bundle. But $M\times E$ is by construction of $\hat Y$
the space of oriented frames in $\hat Y$. Thus $M\times E$ is homotopy equivalent to $\hat X \times G$ as a  principal $G$-bundle.

 We deduce again, that the orbit map $o_p :G\to M$ used in the proof of \tref{main} admits an inverse
$i :M\to G$, such that $i \circ o_p $ is homotopy equivalent to the identity. Again, as above, this provides a contradiction  if
the stabilizer $G_p$ is not trivial. Therefore, $X$ is an $n$-dimensional  manifold.  Since  $X$ is $n$-connected,  it 
is contractible.
\end{proof}

I wonder whether  \pref{dimension} is optimal in high dimensions. In fact, I do not know a single example of a $4$-connected bad orbifold. I would like to finish the note by formulating two problems.

\begin{quest}
Do highly connected bad orbifolds exist?
\end{quest}

\begin{quest}
Does an analogue of \tref{main} hold true for non-smooth orbifolds? Does it hold true for etale groupoids of isometries?
\end{quest}

\noindent\textbf{Acknowledgements}  I am grateful to Mike Davis and Burkhard Wilking for helpful discussions.

\bibliographystyle{alpha}
\bibliography{controrbi}

\begin{thebibliography}{Dav10}

\bibitem[Dav10]{Davis}
M.~Davis.
\newblock Lectures on orbifolds and reflection groups.
\newblock In {\em Transformation groups and moduli spaces of curves}, Higher
  Education Press, pages 63--93. Springer-Verlag, 2010.

\bibitem[Hae84]{Hae}
A.~Haefliger.
\newblock Groupoides {d'holonomie} et classifiants.
\newblock {\em Asterisque}, 116:70--97, 1984.

\end{thebibliography}

\end{document}